\documentclass[11pt,reqno]{amsart}
\usepackage{amsmath,amssymb}

 \makeatletter
 \evensidemargin  \oddsidemargin
 \makeatother

\begin{document}
\thispagestyle{empty}
 \setcounter{page}{1}

\begin{center}
{\large\bf Solution and stability of generalized  mixed type cubic
, quadratic and additive
 functional equation in quasi-Banach spaces

\vskip.20in

{\bf M. Eshaghi Gordji } \\[2mm]

{\footnotesize Department of Mathematics,
Semnan University,\\ P. O. Box 35195-363, Semnan, Iran\\
[-1mm] e-mail: {\tt madjid.eshaghi@gmail.com}}

{\bf H. Khodaie  } \\[2mm]

{\footnotesize Department of Mathematics,
Semnan University,\\ P. O. Box 35195-363, Semnan, Iran\\
[-1mm] e-mail: {\tt khodaie.ham@gmail.com}}}
\end{center}
\vskip 5mm

 \noindent{\footnotesize{\bf Abstract.}
In this paper, we achieve the general solution and the generalized
  Hyers-Ulam-Rassias stability of the following functional equation
$$f(x+ky)+f(x-ky)=k^2f(x+y)+k^2f(x-y)+2(1-k^2)f(x)\eqno \hspace {2 cm}$$for fixed
integers $k$ with $k\neq0,\pm1$ in the quasi-Banach spaces.
 \vskip.10in

  \newtheorem{df}{Definition}[section]
  \newtheorem{rk}[df]{Remark}
   \newtheorem{lem}[df]{Lemma}
   \newtheorem{thm}[df]{Theorem}
   \newtheorem{pro}[df]{Proposition}
   \newtheorem{cor}[df]{Corollary}
   \newtheorem{ex}[df]{Example}

 \setcounter{section}{0}
 \numberwithin{equation}{section}

\vskip .2in

\begin{center}
\section{Introduction}
\end{center}

The stability problem of functional equations originated from a
question of Ulam [20] in 1940, concerning the stability of group
homomorphisms. Let $(G_1,.)$ be a group and let $(G_2,*)$ be a
metric group with the metric $d(.,.).$ Given $\epsilon >0$, dose
there exist a $\delta
>0$, such that if a mapping $h:G_1\longrightarrow G_2$ satisfies the
inequality $d(h(x.y),h(x)*h(y)) <\delta$ for all $x,y\in G_1$,
then there exists a homomorphism $H:G_1\longrightarrow G_2$ with
$d(h(x),H(x))<\epsilon$ for all $x\in G_1?$ In the other words,
Under what condition dose there exists a homomorphism near an
approximate homomorphism? The concept of stability for functional
equation arises when we replace the functional equation by an
inequality which acts as a perturbation of the equation. In 1941,
D. H. Hyers [9] gave a first affirmative  answer to the question
of Ulam for Banach spaces. Let $f:{E}\longrightarrow{E'}$ be a
mapping between Banach spaces such that
$$\|f(x+y)-f(x)-f(y)\|\leq \delta $$
for all $x,y\in E,$ and for some $\delta>0.$ Then there exists a
unique additive mapping $T:{E}\longrightarrow{E'}$ such that
$$\|f(x)-T(x)\|\leq \delta$$
for all $x\in E.$ Moreover if $f(tx)$ is continuous in t for each
fixed $x\in E,$ then $T$ is linear. In 1978, Th. M. Rassias [16]
provided a generalization of Hyers' Theorem which allows the Cauchy
difference to be unbounded.

The functional equation
$$f(x+y)+f(x-y)=2f(x)+2f(y),\eqno \hspace {0.5 cm}(1.1)$$
is related to symmetric bi-additive function [1,2,10,13]. It is
natural that this equation is called a quadratic functional
equation. In particular, every solution of the quadratic equation
(1.1) is said to be a quadratic function. It is well known that a
function $f$ between real vector spaces is quadratic if and only if
there exits a unique symmetric bi-additive function $B$ such that
$f(x)=B(x,x)$ for all $x$ (see [1,13]). The bi-additive function $B$
is given by
$$B(x,y)=\frac{1}{4}(f(x+y)-f(x-y)).\eqno \hspace {0.5 cm}(1.2)$$
A Hyers-Ulam-Rassias stability problem for the quadratic
functional equation (1.1) was proved by Skof for functions
$f:A\longrightarrow B$, where A is normed space and B Banach space
(see [18]). Cholewa [4] noticed that the Theorem of Skof is still
true if relevant domain $A$ is replaced an abelian group. In the
paper [5] , Czerwik proved the Hyers-Ulam-Rassias stability of the
equation (1.1). Grabiec [8] has generalized these result mentioned
above. Jun and Kim [11] introduced the following cubic functional
equation $$f(2x+y)+f(2x-y)=2f(x+y)+2f(x-y)+12f(x), \eqno(1.3)$$
and they established the general solution and the generalized
Hyers-Ulam-Rassias stability for the  functional equation (1.3).
The $f(x)=x^3$ satisfies the functional equation (1.3), which is
called a cubic functional equation. Every solution of the cubic
functional equation is said to be a cubic function.

 Jun and Kim
proved that  a function $f$ between real vector spaces X and Y is
a solution of (1.3) if and only if there exits a unique function
$C:X\times X\times X\longrightarrow Y$ such that $f(x)=C(x,x,x)$
for all $x\in X,$ and $C$ is
symmetric for each fixed one variable and is additive for fixed two variables.\\
K. Jun and H. Kim [12], have obtained the generalized Hyers-Ulam
stability for a mixed type of cubic and additive functional
equation. In addition the generalized Hyers-Ulam-Rassias for a mixed
type of quadratic and additive functional equation in quasi-Banach
spaces have been investigated by A. Najati and M. B. Moghimi [14].
Also A. Najati and G. Zamani Eskandani [15] introduced the following
functional equation
$$f(2x+y)+f(2x-y)=2f(x+y)+2f(x-y)+2f(2x)-4f(x),\eqno (1.4)$$ with
$f(0)=0$. It is easy to see that the mapping $f(x)=ax^3+bx$ is a
solution of the functional equation (1.4). They established the
general solution and the generalized Hyers-Ulam-Rassias stability
for the  functional equation (1.4) whenever $f$ is a mapping
between two quasi-Banach spaces. Now, we introduce the following
functional equation for fixed integers $k$ with $k\neq0,\pm1$:
$$f(x+ky)+f(x-ky)=k^2f(x+y)+k^2f(x-y)+2(1-k^2)f(x), \eqno (1.5)$$with
$f(0)=0$. It is easy to see that the function $f(x)=ax^3+bx^2+cx$
is a solution of the functional equation (1.5). In the present
paper we investigate the general solution of functional equation
(1.5) when $f$ is a mapping between vector spaces, and we
establish the generalized Hyers-Ulam-Rassias stability of the
functional equation (1.5) whenever $f$ is a function between two
quasi-Banach spaces.

 We recall some basic facts concerning quasi-Banach
space and some preliminary results.
\begin{df}\label{t2} (See [3, 17].) Let $X$ be a real linear space.
A quasi-norm is a real-valued function on $X$ satisfying the
following:\\ (1) $\|x\| \geq 0$ for all $x\in X$ and $\|x\|=0$ if
and only if $x=0~.$\\
(2) $\|\lambda.x\|=|\lambda|.\|x\|$ for all $\lambda \in \Bbb R$
and all $x \in X~.$\\
(3) There is a constant $M \geq 1$ such that $\|x+y\| \leq
M(\|x\|+\|y\|)$ for all $x,y \in X~.$

It follows from condition (3) that
$$\|\sum^{2n}_{i=1} x_{i}\| \leq M^n \sum^{2n}_{i=1}\|x_{i}\|,\hspace {1.5 cm}
\|\sum^{2n+1}_{i=1} x_{i}\| \leq M^{n+1}
\sum^{2n+1}_{i=1}\|x_{i}\|$$for all $~n\geq 1$ and all
$~~x_{1},x_{2},....,x_{2n+1}\in X.$
\end{df}
The pair $(X,\|.\|)$ is called a quasi-normed space if $\|.\|$ is
a quasi-norm on $X~.$ The smallest possible $M$ is called the
modulus of concavity of $\|.\|.$ A quasi-Banach space is a
complete quasi-normed space.

 A quasi-norm $\|.\|$ is called a
p-norm $(0 < p \leq 1)$ if
$$\|x+y\|^p \leq \|x\|^p+\|y\|^p,$$
for all $x,y \in X~.$ In this case, a quasi-Banach space is called
a p-Banach space.

Given a p-norm, the formula $d(x,y):=\|x-y\|^p$ gives us a
translation invariant metric on X. By the Aoki-Rolewicz Theorem [
17] (see also [3]), each quasi-norm is equivalent to some p-norm.
Since it is much easier to work with p-norms, henceforth we restrict
our attention mainly to p-norms. More over in [19], J. Tabor has
investigated a version of Hyers-Rassias-Gajda Theorem (see[6,16]) in
quasi-Banach spaces.
\\  \vskip .2in

\section{ General solution}
Throughout this section, $X$ and $Y$  will be  real vector spaces.
Before proceeding the proof of Theorem 2.3 which is the main
result in this section, we shall need the following two Lemmas.
\begin{lem}\label{t2}
If an even function $f:X\longrightarrow Y$ with $f(0)=0$ satisfies
(1.5), then $f$ is quadratic.
\end{lem}
\begin{proof}
Setting $x=0$ in (1.5), by evenness of $f$, we obtain
$f(kx)=k^2f(x).$ Replacing $x$ by $kx$ in (1.5) and then using the
identity $f(kx)=k^2f(x)$, we lead to
$$f(kx+y)+f(kx-y)=f(x+y)+f(x-y)+2(k^2-1)f(x)\eqno \hspace {3.2cm}(2.1)$$
for all $x,y \in X.$ Interchange $x$ with $y$ in (1.5), gives
$$f(y+kx)+f(y-kx)=k^2f(y+x)+k^2f(y-x)+2(1-k^2)f(y)\eqno \hspace {3cm}(2.2)$$
for all $x,y \in X.$ By evenness of  $f$, it follows from  (2.2)
that
$$f(kx+y)+f(kx-y)=k^2f(x+y)+k^2f(x-y)+2(1-k^2)f(y)\eqno \hspace {3cm}(2.3)$$
for all $x,y \in X.$ But, $k\neq 0,\pm1$ so from (2.1) and (2.3),
we obtain
$$f(x+y)+f(x-y)=2f(x)+2f(y)\eqno \hspace {7.4cm}$$
for all $x,y \in X.$ This shows that $f$ is quadratic, which
completes the proof of Lemma.
\end{proof}

\begin{lem}\label{t2} If an odd function $f:X\longrightarrow Y$ satisfies
(1.5), then f is a cubic-additive.
\end{lem}
\begin{proof}
Letting $y=x$ in (1.5), we get by oddness of $f,$
$$f((k+1)x)=f((k-1)x)+k^2f(2x)+2(1-k^2)f(x)\eqno \hspace {5cm}(2.4)\hspace {.1cm}$$
for all $x,y \in X.$ Replacing $x$ by $(k-1)x$ in (1.5), gives
\begin{align*}
f&((k-1)x+ky)+f((k-1)x-ky)\\&=k^2f((k-1)x+y)+k^2f((k-1)x-y)+2(1-k^2)
f((k-1)x)\hspace {2.7cm}(2.5)\hspace {.1cm}
\end{align*}
for all $x,y \in X.$ Now, if we Replacing $x$ by $(k+1)x$ in (1.5)
and using (2.4), we see that
\begin{align*}
f&((k+1)x+ky)+f((k+1)x-ky)\\&=k^2f((k+1)x+y)+k^2f((k+1)x-y)+2(1-k^2)
f((k-1)x)\\&\hspace {.45cm}+2k^2(1-k^2)f(2x)+4(1-k^2)^2f(x)\hspace
{6.15cm}(2.6)
\end{align*}
for all $x,y \in X.$ We substitute $x=x+y$ in (1.5) and then
$x=x-y$ in (1.5) to obtain that
$$f(x+(k+1)y)+f(x-(k-1)y)=k^2f(x+2y)+2(1-k^2)f(x+y)+k^2f(x)\eqno \hspace {.2cm}(2.7)\hspace {.15cm}$$
and
$$f(x-(k+1)y)+f(x+(k-1)y)=k^2f(x-2y)+2(1-k^2)f(x-y)+k^2f(x)\eqno \hspace {1cm}(2.8)\hspace {.15cm}$$
for all $x,y \in X.$ If we subtract (2.8) from (2.7), we have
\begin{align*}
f&(x+(k+1)y)-f(x-(k+1)y)\\&=k^2f(x+2y)-k^2f(x-2y)+f(x+(k-1)y)-f(x-(k-1)y)\\&
\hspace {.45cm}+2(1-k^2)f(x+y)-2(1-k^2)f(x-y)\hspace {5.55cm}(2.9)
\end{align*}
for all $x,y \in X.$ Interchange $x$ with $y$ in (2.9) and using
oddness of $f$, we get the relation
\begin{align*}
f&((k+1)x+y)+f((k+1)x-y)\\&=k^2f(2x+y)+k^2f(2x-y)+f((k-1)x+y)+f((k-1)x-y)\\&
\hspace {.45cm}+2(1-k^2)f(x+y)+2(1-k^2)f(x-y)\hspace
{5.4cm}(2.10)\hspace {.0cm}\end{align*}for all $x,y \in X.$ It
follows  from (2.6) and (2.10) that
\begin{align*}
f&((k+1)x+ky)+f((k+1)x-ky)\\&=k^2f((k-1)x+y)+k^2f((k-1)x-y)+k^4f(2x+y)
+k^4f(2x-y)\\&
\hspace{.45cm}+2k^2(1-k^2)f(x+y)+2k^2(1-k^2)f(x-y)\\&
\hspace{.45cm}+2(1-k^2)f((k-1)x)
+2k^2(1-k^2)f(2x)+4(1-k^2)^2f(x)\hspace {2.6cm}(2.11)
\end{align*}
for all $x,y \in X.$ We substitute $y=x+y$ in (1.5) and then
$y=x-y$ in (1.5), we get by the oddness of $f,$
$$f((k+1)x+ky)-f((k-1)x+ky)=k^2f(2x+y)+k^2f(-y)+2(1-k^2)f(x)
\eqno \hspace {.7cm}(2.12)\hspace {.15cm}$$ and
$$f((k+1)x-ky)-f((k-1)x-ky)=k^2f(2x-y)+k^2f(y)+2(1-k^2)f(x)
\eqno \hspace {.9cm}(2.13)\hspace {.2cm}$$for all $x,y \in X.$
 Then, by adding (2.12) to (2.13) and then using (2.5), we lead to
\begin{align*}
f&((k+1)x+ky)+f((k+1)x-ky)\\&=k^2f((k-1)x+y)+k^2f((k-1)x-y)\\&
\hspace {.4cm}+k^2f(2x+y) +k^2f(2x-y)+4(1-k^2)f(x)\hspace
{4.8cm}(2.14)\hspace {.1cm}
\end{align*}
for all $x,y \in X.$ Finally, if we compare (2.11) with (2.14),
then we conclude that
$$f(2x+y)+f(2x-y)=2f(x+y)+2f(x-y)+2(f(2x)-2f(x))\eqno \hspace {3.3cm}$$
for all $x,y \in X.$ Hence, $f$ is cubic-additive function
(see[15]). This
 completes the proof of Lemma.
\end{proof}

\begin{thm}\label{t2} A function $f:X\rightarrow Y$ with $f(0)=0$ satisfies
(1.5) for all $x,y\in X$ if and only if there exist functions
$C:X\times X\times X\longrightarrow Y$ and $B:X\times
X\longrightarrow Y$ and $A:X\rightarrow Y,$ such that
$f(x)=C(x,x,x)+B(x,x)+A(x)$ for all $x\in X,$ where the function
 $C$ is symmetric for each fixed one variable and is additive for
fixed two variables and $B$ is symmetric bi-additive and $A$ is
additive.
\end{thm}

\begin{proof} Let $f$ with $f(0)=0$ satisfies (1.5). We decompose $f$ into the even part and odd
part by putting
$$f_e(x)=\frac{1}{2}(f(x)+f(-x)),~~\hspace {0.3 cm}f_o(x)=\frac{1}{2}(f(x)-f(-x)),$$
for all $x\in X.$ It is clear that $f(x)=f_e(x)+f_o(x)$ for all
$x\in X.$ It is easy to show that the functions $f_e$ and $f_o$
satisfy (1.5). Hence by Lemmas 2.1 and 2.2, we achieve that the
functions $f_e$ and $f_o$ are quadratic and cubic-additive,
respectively, thus there exist a symmetric bi- additive function
$B:X\times X\longrightarrow Y$ such that $f_e(x)=B(x,x)$ for all
$x\in X,$ and the function $C:X\times X\times X\longrightarrow Y$
and additive function $A:X\rightarrow Y$ such that
$f_{o}(x)=C(x,x,x)+A(x),$ for all $x\in X,$ where the function
 $C$ is symmetric for each fixed one variable and is additive for
fixed two variables. Hence, we get $f(x)=C(x,x,x)+B(x,x)+A(x),$ for
all $x\in X.$

Conversely, let $f(x)=C(x,x,x)+B(x,x)+A(x)$ for all $x\in X,$ where
the function $C$ is symmetric for each fixed one variable and is
additive for fixed two variables and $B$ is bi- additive and $A$ is
additive. By a simple computation one can show that the functions
$x\mapsto C(x,x,x)$ and $x\mapsto B(x,x)$ and $A$ satisfy the
functional equation (1.5). So the function $f$ satisfies (1.5).
\end{proof}
\section{ Stability  }
Throughout this section, assume that $X$ quasi-Banach space with
quasi-norm $\|.\|_{X}$ and that $Y$ is a p-Banach space  with
p-norm $\|.\|_{Y}.$ Let $M$ be the modulus of concavity of
$\|.\|_{Y}.$

In this section, using an idea of G$\check{a}$vruta [7] we prove the
stability of Eq.(1.5) in the spirit of Hyers, Ulam and Rassias. We
need the following Lemma in the main Theorems. Now before taking up
the main subject, given $f:X\rightarrow Y$, we define the difference
operator $D_f:X\times X \rightarrow Y$ by
$$D_{f}(x,y)=f(x+ky)+f(x-ky)-k^2f(x+y)-k^2f(x-y)-2(1-k^2)f(x)$$
for all $x,y \in X.$

\begin{lem}\label{t'2}(see [14]) Let $0<p\leq1$ and let
$x_1,x_2,\ldots,x_n$ be non-negative real numbers. Then
$$(\sum^{n}_{i=1} x_i)^p\leq \sum^{n}_{i=1} {x_i}^p. $$
\end{lem}

\begin{thm}\label{t2}
Let $j\in \{-1,1\}$ be fixed and let $\varphi:X\times X\rightarrow
[0,\infty)$ be a function such that
$$\lim_{n\rightarrow\infty} k^{2nj}
\varphi(\frac{x}{k^{nj}},\frac{y}{k^{nj}})=0 \eqno\hspace
{2cm}(3.1)$$ for all $x,y\in X$ and
$$\tilde{\psi}_e(x):=\sum^{\infty}_{i=\frac{1+j}{2}} k^{2ipj}
\varphi^p(0,\frac{x}{k^{ij}})<\infty \eqno \hspace {2cm}(3.2)$$
for all $x\in X.$ Suppose that an even function $f:X\rightarrow Y$
with $f(0)=0$  satisfies the inequality
$$\| D_f(x,y)\|_Y \leq\varphi(x,y)\eqno \hspace {2cm} (3.3)$$
for all $x,y\in X.$ Then the limit
$$Q(x):=\lim_{n\rightarrow\infty} k^{2nj}
f(\frac{x}{k^{nj}}) \eqno \hspace {2cm}(3.4)$$ exists for all
$x\in X$ and $Q:X\rightarrow Y$ is a unique quadratic function
satisfying
$$\|f(x)-Q(x)\|_Y \leq\frac{M}{2k^2}[\tilde{\psi}_e(x)]^\frac{1}{p}\eqno\hspace {2cm}(3.5)$$
for all $x\in X.$
\end{thm}
\begin{proof}
Let $j=1.$ By putting $x=0$ in (3.3), we get
$$\|2f(ky)-2k^2f(y)\|_Y\leq\varphi(0,y) \eqno\hspace {2cm}(3.6)$$ for all $y\in X.$
If we replace $y$ in (3.6) by $x,$ and divide both sides of
$(3.6)$ by $2,$ we get
$$\|f(kx)-k^2f(x)\|_Y\leq\frac{1}{2}\varphi(0,x) \eqno\hspace {2cm}(3.7)$$ for all $x\in X.$
Let $\psi_{e}(x)=\frac{1}{2}\varphi(0,x)$ for all $x\in X,$ then
by $(3.7),$ we get
$$\|f(kx)-k^2f(x)\|_Y\leq\psi_{e}(x) \eqno\hspace {2cm}(3.8)$$ for all $x\in X.$
If we replace $x$ in (3.8) by $\frac{x}{k^{n+1}}$ and multiply
both sides of (3.8) by $k^{2n},$ then we have
$$\|k^{2(n+1)}f(\frac{x}{k^{n+1}})-k^{2n}f(\frac{x}{k^n})\|_Y\leq M k^{2n}\psi_{e}(\frac{x}{k^{n+1}}) \eqno\hspace {2cm}(3.9)$$
for all $x\in X$ and all non-negative integers $n$. Since $Y$ is
p-Banach space, then by (3.9) gives
\begin{align*}
\|k^{2(n+1)}f(\frac{x}{k^{n+1}})-k^{2m}
f(\frac{x}{k^m})\|_Y^p&\leq
\sum^{n}_{i=m}\|k^{2(i+1)}f(\frac{x}{k^{i+1}})-k^{2i}
f(\frac{x}{k^i})\|_Y^p\\
&\leq M^p \sum^{n}_{i=m} k^{2ip} {\psi_{e}}^p (\frac{x}{k^{i+1}})
\hspace {3.9cm}(3.10)
\end{align*}
for all non-negative integers $n$ and $m$ with $n\geq m$ and all
$x\in X.$ Since ${\psi_{e}}^p(x)=\frac{1}{2^p}{\varphi}^p(0,x)$
for all $x\in X,$ therefore  by (3.2)  we have
$$\sum^{\infty}_{i=1} k^{2ip} {\psi_{e}}^p(\frac{x}{k^i})<\infty \eqno\hspace
{2cm}(3.11)$$ for all $x\in X.$ Therefore we conclude from (3.10)
and (3.11) that the sequence $\{k^{2n}f(\frac{x}{k^n})\}$ is a
Cauchy sequence for all $x\in X.$ Since $Y$ is complete, the
sequence $\{k^{2n}f(\frac{x}{k^n})\}$ converges for all $x\in X.$
So one can define the function $Q:X\rightarrow Y$ by (3.4) for all
$x\in X.$ Letting $m=0$ and passing the limit $n\rightarrow\infty$
in $(3.10),$ we get
$$\|f(x)-Q(x)\|_Y^p\leq M^p\sum^{\infty}_{i=0}k^{2ip}{\psi_{e}}^p
(\frac{x}{k^{i+1}})
=\frac{M^p}{k^{2p}}\sum^{\infty}_{i=1}k^{2ip}{\psi_{e}}^p(\frac{x}{k^i})
\eqno\hspace {2cm}(3.12)$$ for all $x\in X.$ Therefore (3.5)
follows from (3.2) and (3.12). Now we show that $Q$ is quadratic.
It follows from (3.1), (3.3) and (3.4)
$$\|D_Q(x,y)\|_Y=\lim_{n\rightarrow\infty} k^{2n} \|D_f(\frac{x}{k^n},\frac{y}{k^n})\|_Y
\leq\lim_{n\rightarrow\infty} k^{2n}
\varphi(\frac{x}{k^n},\frac{y}{k^n})=0$$ for all $x,y\in X.$
Therefore the function $Q:X\rightarrow Y$ satisfies (1.5). Since $f$
is an even function, then (3.4) implies that the function $Q:X \to
Y$ is even. Therefore by Lemma 2.1, we get that the
function $Q:X \to Y$ is quadratic.\\
To prove the uniqueness of $Q,$ let $Q^{'}:X\rightarrow Y$ be
another quadratic function satisfying (3.5). Since
$$\lim_{n\rightarrow\infty} k^{2np}\sum^{\infty}_{i=1}k^{2ip}\varphi^p(0,\frac{x}{k^{i+n}})
=\lim_{n\rightarrow\infty}
\sum^{\infty}_{i={n+1}}k^{2ip}\varphi^p(0,\frac{x}{k^i})=0$$ for
all $x\in X,$ then $$\lim_{n\rightarrow\infty} k^{2np}
\tilde{\psi}_e(\frac{x}{k^n})=0 \eqno\hspace {2.5cm}$$ for all
$x\in X.$ Therefore it follows from (3.5) and the last equation
that
$$\|Q(x)-Q^{'}(x)\|_Y^p=\lim_{n\rightarrow\infty} k^{2np}
\|f(\frac{x}{k^n})-Q^{'}(\frac{x}{k^n})\|_Y^p\\
\leq\frac{M^p}{{2k^2}^p}\lim_{n\rightarrow\infty}
k^{2np}\tilde{\psi}_e(\frac{x}{k^n})=0$$ for all $x\in X.$ Hence
$Q=Q^{'}.$\\  For $j=-1$, we can prove the Theorem by a similar
technique.
\end{proof}
\begin{cor}\label{t2}
Let $\theta, r, s$ be non-negative real numbers such that $r,s> 2$
or $0\leq r,s<2$. Suppose that an even function $f:X\rightarrow Y$
with $f(0)=0$ satisfies the inequality
$$\| D_f(x,y)\|_Y \leq\theta(\|x\|_X^r+\|y\|_X^s) ,\eqno \hspace {3cm} (3.13)$$
for all $x,y\in X.$ Then there exists a unique quadratic function
$Q:X\rightarrow Y$ satisfies
$$\|f(x)-Q(x)\|_Y \leq\frac{M\theta}{2}~\textbf{(}\frac{1}{|k^{2p}-k^{sp}|}~\|x\|_X^{sp}\textbf{)
}^{\frac{1}{p}}\eqno\hspace {2cm}$$
for all $x\in X.$
\end{cor}
\begin{proof}
It follows from Theorem 3.2 by putting
$\varphi(x,y):=\theta(\|x\|_X^r+\|y\|_X^s)$ for all $x,y\in X.$
\end{proof}
\begin{thm}\label{t2}
Let $j\in \{-1,1\}$ be fixed and let $\varphi_a:X\times
X\rightarrow [0,\infty)$ be a function such that
$$\lim_{n\rightarrow\infty} 2^{nj}
\varphi_a(\frac{x}{2^{nj}},\frac{y}{2^{nj}})=0 \eqno\hspace
{2cm}(3.14)$$ for all $x,y\in X$ and
$$\sum^{\infty}_{i=\frac{1+j}{2}} 2^{ipj}
{\varphi_a}^p(\frac{x}{2^{ij}},\frac{y}{2^{ij}})<\infty \eqno
\hspace {2cm}(3.15)$$ for all $x\in X$ and for all $y\in
\{x,2x,3x\}.$ Suppose that an odd function $f:X\rightarrow Y$
satisfies the inequality
$$\| D_f(x,y)\|_Y \leq\varphi_a(x,y)\eqno \hspace {4.5cm} (3.16)$$
for all $x,y\in X.$ Then the limit
$$A(x):=\lim_{n\rightarrow\infty} 2^{nj}
[f(\frac{x}{2^{nj-1}})-8f(\frac{x}{2^{nj}})] \eqno \hspace
{2.4cm}(3.17)$$ exists for all $x\in X$ and $A:X\rightarrow Y$ is
a unique additive function satisfying
$$\|f(2x)-8f(x)-A(x)\|_Y \leq\frac{M^5}{2}[\widetilde{\psi}_a(x)]^\frac{1}{p}\eqno\hspace {2cm}(3.18)$$
for all $x\in X,$ where
\begin{align*}
\widetilde{\psi}_a(x):=\sum^{\infty}_{i=\frac{1+j}{2}}
&2^{ipj}~\textbf{\{}\frac{1}{k^{2p}(1-k^2)^p}
~[~(5-4k^2)^p{\varphi_a}^p(\frac{x}{2^{ij}},\frac{x}{2^{ij}})+k^{2p}
{\varphi_a}^p(\frac{2x}{2^{ij}},\frac{2x}{2^{ij}})\\&+(2k^2)^p
{\varphi_a}^p(\frac{2x}{2^{ij}},\frac{x}{2^{ij}})+
{\varphi_a}^p(\frac{x}{2^{ij}},\frac{3x}{2^{ij}})+(4-2k^2)^p
{\varphi_a}^p(\frac{x}{2^{ij}},\frac{2x}{2^{ij}})\\&+2^p
{\varphi_a}^p(\frac{(1+k)x}{2^{ij}},\frac{x}{2^{ij}}) +2^p
{\varphi_a}^p(\frac{(1-k)x}{2^{ij}},\frac{x}{2^{ij}})
\\&+{\varphi_a}^p(\frac{(1+2k)x}{2^{ij}},\frac{x}{2^{ij}}) +
{\varphi_a}^p(\frac{(1-2k)x}{2^{ij}},\frac{x}{2^{ij}})~]~
\textbf{\}}. \hspace{3cm}(3.19)\end{align*}
\end{thm}
\begin{proof}
Let $j=1.$ By replacing $y$ by $x$ in (3.16), we have
$$\|f((1+k)x)+f((1-k)x)-k^2f(2x)-2(1-k^2)f(x)\|\leq
\varphi_a(x,x) \eqno \hspace {2.5cm}(3.20)$$
 for all $x\in X.$ It follows from (3.20) that
$$\|f(2(1+k)x)+f(2(1-k)x)-k^2f(4x)-2(1-k^2)f(2x)\|\leq
\varphi_a(2x,2x) \eqno \hspace {2.3cm}(3.21)$$
 for all $x\in X.$ Replacing $x$ and $y$ by $2x$ and $x$ in
 (3.16), respectively, we get
$$\|f((2+k)x)+f((2-k)x)-k^2f(3x)-2(1-k^2)f(2x)-k^2f(x)\|\leq
\varphi_a(2x,x) \eqno \hspace {1.1cm}(3.22)$$ for all $x\in X.$
Letting $y$ by $2x$ in (3.16) gives
$$\|f((1+2k)x)+f((1-2k)x)-k^2f(3x)-k^2f(-x)-2(1-k^2)f(x)\|\leq
\varphi_a(x,2x) \eqno \hspace {1cm}(3.23)$$
 for all $x\in X.$
putting $y$ by $3x$ in (3.16), we obtain
$$\|f((1+3k)x)+f((1-3k)x)-k^2f(4x)-k^2f(-2x)-2(1-k^2)f(x)\|\leq
\varphi_a(x,3x) \eqno \hspace {1cm}(3.24)$$
 for all $x\in X.$ Replacing $x$ and $y$ by $(1+k) x$ and $x$ in
 (3.16), respectively, we get
\begin{align*}\|f((1+2k)x)+f(x)-k^2f((2+k)x)-k^2f(kx)-2(1-k^2)&f((1+k)x)\|\\&\leq
\varphi_a((1+k)x,x) \hspace {1cm}(3.25) \end{align*}
 for all $x\in X.$
 Replacing $x$ and $y$ by $(1-k)x$ and $x$ in
 (3.16), respectively, one gets
\begin{align*}\|f((1-2k)x)+f(x)-k^2f((2-k)x)-k^2f(-kx)-2(1-k^2&)f((1-k)x)\|\\&\leq
\varphi_a((1-k)x,x) \hspace {0.8cm}(3.26) \end{align*}
 for all $x\in X.$
 Replacing $x$ and $y$ by $(1+2k)x$ and $x$ in
 (3.16), respectively, we obtain
\begin{align*}\|f((1+3k)x)+f((1+k)x)-k^2f(2(1+k)x)-k^2f(2kx)-&2(1-k^2)f((1+2k)x)\|\\&\leq
\varphi_a((1+2k)x,x) \hspace {0.9cm}(3.27) \end{align*}
 for all $x\in X.$
 Replacing $x$ and $y$ by $(1-2k)x$ and $x$ in
 (3.16), respectively, we have
\begin{align*}\|f((1-3k)x)+f((1-k)x)-k^2f(2(1-k)x)-k^2f(-2kx&)-2(1-k^2)f((1-2k)x)\|\\&\leq
\varphi_a((1-2k)x,x) \hspace {1cm}(3.28) \end{align*}
 for all $x\in X.$ It follows from (3.25), (3.26) and
 oddness $f$ that
\begin{align*}\|&f((1+2k)x)+f((1-2k)x)+2f(x)-k^2f((2+k)x)-k^2f((2-k)x)
\\&-2(1-k^2)f((1+k)x)-2(1-k^2)f((1-k)x)\|\\& \hspace {5cm}\leq
M(\varphi_a((1+k)x,x) +\varphi_a((1-k)x,x)) \hspace {1cm}(3.29)
\end{align*}
 for all $x\in X.$ Now, from (3.20), (3.22), (3.23) and (3.29), we conclude that
\begin{align*}\|f(3x)-4f(2x)+5f(x)\|\leq
&\frac{M^3}{k^2(1-k^2)}~[2(1-k^2)\varphi_a(x,x)+k^2\varphi_a(2x,x)\\&+\varphi_a(x,2x)
+\varphi_a((1+k)x,x)+\varphi_a((1-k)x,x)~] \hspace {1.3cm}(3.30)
\end{align*}
 for all $x\in X.$ On the other hand it follows from (3.27), (3.28) and
 oddness $f$ that
\begin{align*}\|&f((1+3k)x)+f((1-3k)x)+f((1+k)x)+f((1-k)x)-k^2f(2(1+k)x)\\&-k^2f(2(1-k)x)
-2(1-k^2)f((1+2k)x)-2(1-k^2)f((1-2k)x)\|\\& \hspace {5cm}\leq
M(\varphi_a((1+2k)x,x) +\varphi_a((1-2k)x,x)) \hspace
{1.1cm}(3.31)
\end{align*}
 for all $x\in X.$ Also, from (3.20), (3.21), (3.23), (3.24) and (3.31), we
 lead to
\begin{align*}\|f&(4x)-2f(3x)-2f(2x)+6f(x)\|\leq
\frac{M^3}{k^2(1-k^2)}~[\varphi_a(x,x)+k^2\varphi_a(2x,2x)\\&+2(1-k^2)\varphi_a(x,2x)
+\varphi_a(x,3x)+\varphi_a((1+2k)x,x)+\varphi_a((1-2k)x,x)~]
\hspace {1.6cm}(3.32)
\end{align*}
 for all $x\in X.$ Finally, by using (3.30) and (3.32), we obtain
 that
\begin{align*}\|f&(4x)-10f(2x)+16f(x)\|\leq
\frac{M^5}{k^2(1-k^2)}~[(5-4k^2)\varphi_a(x,x)+k^2\varphi_a(2x,2x)\\&+2k^2\varphi_a(2x,x)
+(4-2k^2)\varphi_a(x,2x)+\varphi_a(x,3x)
+2\varphi_a((1+k)x,x)\\&+2\varphi_a((1-k)x,x)
+\varphi_a((1+2k)x,x)+\varphi_a((1-2k)x,x)~] \hspace {3.6cm}(3.33)
\end{align*}
 for all $x\in X,$ and let
\begin{align*}\psi_{a}(x)=&
\frac{1}{k^2(1-k^2)}~[(5-4k^2)\varphi_a(x,x)+k^2\varphi_a(2x,2x)\\&+2k^2\varphi_a(2x,x)
+(4-2k^2)\varphi_a(x,2x)+\varphi_a(x,3x)
+2\varphi_a((1+k)x,x)\\&+2\varphi_a((1-k)x,x)
+\varphi_a((1+2k)x,x)+\varphi_a((1-2k)x,x)~]  \hspace
{2.8cm}(3.34)\end{align*} for all $x \in X.$ Therefore $(3.33)$
means that
$$\|f(4x)-10f(2x)+16f(x)\|\leq M^5\psi_{a}(x) \eqno \hspace {6cm}(3.35)$$
for all $x \in X.$ Letting $g:X \to Y$ be a function defined by
$g(x):=f(2x)-8f(x)$ then, we conclude that
$$\|g(2x)-2g(x)\|\leq M^5\psi_{a}(x) \eqno\hspace {7.8cm}(3.36)$$ for all $x \in X.$
If we replace $x$ in (3.36) by $\frac{x}{2^{n+1}}$ and multiply both
sides of (3.36) by $2^n,$ we get
$$\|2^{n+1}g(\frac{x}{2^{n+1}})-2^ng(\frac{x}{2^n})\|_Y\leq M^5 2^n \psi_{a}(\frac{x}{2^{n+1}})
 \eqno\hspace {5.1cm}(3.37)$$
for all $x\in X$ and all non-negative integers $n$. Since $Y$ is
p-Banach space, therefore by inequality (3.37), gives
\begin{align*}
\|2^{n+1}g(\frac{x}{2^{n+1}})-2^m g(\frac{x}{2^m})\|_Y^p&\leq
\sum^{n}_{i=m}\|2^{i+1}g(\frac{x}{2^{i+1}})-2^i
g(\frac{x}{2^i})\|_Y^p\\
&\leq M^{5p} \sum^{n}_{i=m} 2^{ip} {\psi_{a}}^p
(\frac{x}{2^{i+1}}) \hspace {4.4cm}(3.38)
\end{align*}
for all non-negative integers $n$ and $m$ with $n\geq m$ and all
$x\in X.$ Since $0< p\leq1$, then by Lemma 3.1, we get from (3.34),
\begin{align*}{\psi_{a}}^p(x)\leq&
\frac{1}{k^{2p}(1-k^2)^p}~[(5-4k^2)^p{\varphi_a}^p(x,x)+k^{2p}{\varphi_a}^p(2x,2x)
\\&+(2k^2)^p{\varphi_a}^p(2x,x)
+(4-2k^2)^p{\varphi_a}^p(x,2x)+{\varphi_a}^p(x,3x)
+2^p{\varphi_a}^p((1+k)x,x)\\&+2^p{\varphi_a}^p((1-k)x,x)
+{\varphi_a}^p((1+2k)x,x)+{\varphi_a}^p((1-2k)x,x)~] , \hspace
{1.75cm}(3.39)\end{align*}
 for all $x\in X.$ Therefore  it follows from (3.15) and (3.39) that
$$\sum^{\infty}_{i=1} 2^{ip} {\psi_{a}}^p(\frac{x}{2^i})<\infty\hspace
{9.27cm}(3.40)\hspace {.2cm}$$ for all $x\in X.$ Therefore we
conclude from (3.38) and (3.40) that the sequence
$\{2^{n}g(\frac{x}{2^n})\}$ is a Cauchy sequence for all $x\in X.$
Since $Y$ is complete, the sequence $\{2^{n}g(\frac{x}{2^n})\}$
converges for all $x\in X.$ So one can define the mapping
$A:X\rightarrow Y$ by $$A(x)=\lim_{n \to
\infty}2^ng(\frac{x}{2^n}) \hspace {9.1cm}(3.41)\hspace {.2cm}$$
for all $x\in X.$ Letting $m=0$ and passing the limit
$n\rightarrow\infty$ in (3.38), we get
$$\|g(x)-A(x)\|_Y^p\leq M^{5p}\sum^{\infty}_{i=0}2^{ip}{\psi_{a}}^p
(\frac{x}{2^{i+1}})
=\frac{M^{5p}}{2^p}\sum^{\infty}_{i=1}2^{ip}{\psi_{a}}^p(\frac{x}{2^i})
\hspace {3.25cm}(3.42)\hspace {.2cm}$$ for all $x\in X.$ Therefore
(3.18) follows from (3.15) and (3.42). Now we show that $A$ is
additive. It follows from (3.14), (3.37) and (3.41) that
\begin{align*}
\|A(2x)-2A(x)\|_Y &=\lim_{n \to \infty}\|2^n
g(\frac{x}{2^{n-1}})-2^{n+1}g(\frac{x}{2^n})\|_Y\\
&= 2 \lim_{n \to \infty} \|2^{n-1}g(\frac{x}{2^{n-1}})-2^n
g(\frac{x}{2^n})\|_Y \\
& \leq M^5 \lim_{n \to \infty} 2^n \psi_{a}(\frac{x}{2^n})=0
\hspace {6cm}\end{align*} for all $x \in X.$ So
$$A(2x)=2A(x)\hspace {9.8cm}(3.43)$$ for all $x \in X.$ On the
other hand it follows from (3.14), (3.16) and (3.17) that
\begin{align*}
\|D_A(x,y)\|_Y&=\lim_{n \to \infty} 2^n
\|D_g(\frac{x}{2^{n}},\frac{y}{2^{n}})\|_Y=\lim_{n \to \infty} 2^n
\|D_f(\frac{x}{2^{n-1}},\frac{y}{2^{n-1}})-8
D_f(\frac{x}{2^n},\frac{y}{2^n})\|_Y\\
&\leq M^5\lim_{n \to \infty}
2^n\{\|D_f(\frac{x}{2^{n-1}},\frac{y}{2^{n-1}})\|_Y+8
\|D_f(\frac{x}{2^n},\frac{y}{2^n})\|_Y\}\\
&\leq M^5\lim_{n \to \infty}
2^n\{\varphi_a(\frac{x}{2^{n-1}},\frac{y}{2^{n-1}})+8
\varphi_a(\frac{x}{2^n},\frac{y}{2^n})\}=0
\end{align*}
for all $x,y \in X.$ Hence the function $A$ satisfies $(1.5).$ By
Lemma 2.2, the function $x \rightsquigarrow A(2x)-2A(x)$ is
additive. Hence, (3.43) implies that the function $A$ is additive.\\
To prove the uniqueness property of $A,$ let $A^{'}:X \to Y$ be
another additive function satisfying (3.18). Since $$\lim_{n \to
\infty}2^{np}\sum_{i=1}^{\infty}2^{ip}{\varphi_a}^p(\frac{x}{2^{n+i}},\frac{x}{2^{n+i}})
=\lim_{n\to\infty}\sum_{i=n+1}^{\infty}2^{ip}{\varphi_a}^p
(\frac{x}{2^i},\frac{x}{2^i})=0\eqno\hspace{2cm}$$ for all $x\in X$
and for all $y\in \{x, 2x, 3x\},$ then
$$\lim_{n \to \infty}2^{np}\widetilde{\psi}_a(\frac{x}{2^n})=0\hspace{9cm}(3.44)$$
for all $x \in X$. It follows from (3.18) and (3.44) that
$$\|A(x)-A^{'}(x)\|_Y=\lim_{n \to \infty}2^{np}{\|g(\frac{x}{2^n})-A^{'}(\frac{x}{2^n})\|_Y}^p
\leq \frac{M^{5p}}{2^p}\lim_{n \to
\infty}2^{np}\widetilde{\psi}_a(\frac{x}{2^n})=0$$
for all $x \in X.$ So $A=A^{'}.$\\
For $j=-1$, we can prove the Theorem by a similar technique.
\end{proof}

\begin{cor}\label{t2}
Let $\theta, r, s$ be non-negative real numbers such that $r,s> 1$
or $0\leq r,s<1$. Suppose that an odd function $f:X\rightarrow Y$
satisfies the inequality
$$\| D_f(x,y)\|_Y \leq\left\{%
\begin{array}{ll}
    \theta, & \hbox{r =s =0;} \\
    \theta\|x\|_X^r, & \hbox{r $>$ 0, s=0;} \\
    \theta\|y\|_X^s, & \hbox{r=0, s $>$ 0;} \\
    \theta(\|x\|_X^r+\|y\|_X^s), & \hbox{r, s $>$ 0.} \\
\end{array}%
\right.\eqno\hspace{3.5cm}(3.45)$$ for all $x, y \in X.$ Then
there exists a unique additive function $A:X\rightarrow Y$
satisfying
$$\| f(2x)-8f(x)-A(x)\|_Y \leq \frac{M^5 \theta}{k^2(1-k^2)}
\left\{%
\begin{array}{ll}
    \delta_{a}, & \hbox{r =s =0;} \\
    \alpha_{a}~\|x\|_X^r, & \hbox{r $>$ 0, s=0;} \\
    \beta_{a}~\|x\|_X^s, & \hbox{r=0, s $>$ 0;} \\
    (\alpha_{a}^p~\|x\|_X^{rp}+\beta_{a}^p~\|x\|_X^{sp})^{\frac{1}{p}}, & \hbox{r, s $>$ 0.} \\
\end{array}%
\right.\eqno\hspace{4cm}$$ for all $x \in X,$ where
$$\delta_{a}=\textbf{\{}~\frac{1}{2^p-1}~[(5-4k^2)^p+(4-2k^2)^p+k^{2p}(2^p+1)
+2^{p+1}+3]~\textbf{\}}^\frac{1}{p},\eqno\hspace{2cm}$$
\begin{align*}\alpha_{a}=\textbf{\{}~\frac{1}{|2^p-2^{rp}|}~[(5-4k^2)^p+(4-2k^2)^p+&(1+2k)^{rp}
+(1-2k)^{rp}+2^p(1+k)^{rp}\\&+2^p(1-k)^{rp}+2^{rp}k^{2p}(2^p+1)+1]~\textbf{\}}^\frac{1}{p}
,\end{align*}

\begin{align*}\beta_{a}=\textbf{\{}~\frac{1}{|2^p-2^{sp}|}~[(5-4k^2)^p+2^{sp}(4-2k^2)^p+
k^{2p}(2^{sp}+2^p)+3^{sp}+2^{p+1}+2]~\textbf{\}}^\frac{1}{p}
.\end{align*}

\end{cor}

\begin{proof}
It follows from Theorem 3.4 by putting  $\varphi_{a}(x,
y):=\theta(\|x\|_X^r+\|y\|_X^s)$ for all $x, y \in X.$
\end{proof}

\begin{cor}\label{t2}
Let $\theta\geq0$ and $r, s>0$ be non-negative real numbers such
that $\lambda:=r+s\neq1$. Suppose that an odd function
$f:X\rightarrow Y$ satisfies the inequality
$$\| D_f(x,y)\|_Y \leq\theta\|x\|_X^r\|y\|_X^s ,\eqno \hspace {3cm} (3.46)$$
for all $x, y \in X.$ Then there exists a unique additive function
 $A:X\rightarrow Y$ satisfying
$$\|f(2x)-8f(x)-A(x)\|_Y \leq\frac{M^5\theta}{k^2(1-k^2)}~~\varepsilon_{a}~\|x\|_X^\lambda
,\eqno \hspace {5cm} $$
for all $x \in X,$ where
\begin{align*}\varepsilon_{a}=\textbf{\{}~\frac{1}{|2^p-2^{\lambda p}|}~[(5-4k^2)^p&+
2^{sp}(4-2k^2)^p+(1+2k)^{rp}+(1-2k)^{rp}+2^p(1+k)^{rp}
\\&+2^p(1-k)^{rp}+k^{2p}(2^{\lambda p}+2^{(r+1)p})+3^{sp}]~\textbf{\}}^\frac{1}{p}
\end{align*}
for all $x \in X.$
\end{cor}

\begin{proof}
It follows from  Theorem 3.4 by putting $\varphi_{a}(x,
y):=\theta\|x\|_X^r \|y\|_X^s$ for all $x, y \in X.$

\end{proof}

\begin{thm}\label{t2}
Let $j\in \{-1,1\}$ be fixed and let $\varphi_c:X\times
X\rightarrow [0,\infty)$ be a function such that
$$\lim_{n\rightarrow\infty} 8^{nj}
\varphi_c(\frac{x}{2^{nj}},\frac{y}{2^{nj}})=0 \eqno\hspace
{2cm}(3.47)$$ for all $x,y\in X$ and
$$\sum^{\infty}_{i=\frac{1+j}{2}} 8^{ipj}
{\varphi_c}^p(\frac{x}{2^{ij}},\frac{y}{2^{ij}})<\infty \eqno
\hspace {2cm}(3.48)$$ for all $x\in X$ and for all $y\in
\{x,2x,3x\}.$ Suppose that an odd function $f:X\rightarrow Y$
satisfies the inequality
$$\| D_f(x,y)\|_Y \leq\varphi_c(x,y)\eqno \hspace {4.5cm} (3.49)$$
for all $x,y\in X.$ Then the limit
$$C(x):=\lim_{n\rightarrow\infty} 8^{nj}
[f(\frac{x}{2^{nj-1}})-2f(\frac{x}{2^{nj}})] \eqno \hspace
{2.5cm}(3.50)$$ exists for all $x\in X$ and $C:X\rightarrow Y$ is
a unique cubic function satisfying
$$\|f(2x)-2f(x)-C(x)\|_Y \leq\frac{M^5}{8}[\widetilde{\psi}_c(x)]^\frac{1}{p}\eqno\hspace {2cm}(3.51)$$
for all $x\in X,$ where
\begin{align*}
\widetilde{\psi}_c(x):=\sum^{\infty}_{i=\frac{1+j}{2}}
&8^{ipj}~\textbf{\{}\frac{1}{k^{2p}(1-k^2)^p}
~[~(5-4k^2)^p{\varphi_c}^p(\frac{x}{2^{ij}},\frac{x}{2^{ij}})+k^{2p}
{\varphi_c}^p(\frac{2x}{2^{ij}},\frac{2x}{2^{ij}})\\&+(2k^2)^p
{\varphi_c}^p(\frac{2x}{2^{ij}},\frac{x}{2^{ij}})+
{\varphi_c}^p(\frac{x}{2^{ij}},\frac{3x}{2^{ij}})+(4-2k^2)^p
{\varphi_c}^p(\frac{x}{2^{ij}},\frac{2x}{2^{ij}})\\&+2^p
{\varphi_c}^p(\frac{(1+k)x}{2^{ij}},\frac{x}{2^{ij}}) +2^p
{\varphi_c}^p(\frac{(1-k)x}{2^{ij}},\frac{x}{2^{ij}})
\\&+{\varphi_c}^p(\frac{(1+2k)x}{2^{ij}},\frac{x}{2^{ij}}) +
{\varphi_c}^p(\frac{(1-2k)x}{2^{ij}},\frac{x}{2^{ij}})~]~
\textbf{\}}. \hspace{3.1cm}(3.52)\end{align*}
\end{thm}
\begin{proof}
Let $j=1.$ Similar to the proof of Theorem 3.4, we have
$$\|f(4x)-10f(2x)+16f(x)\|\leq M^5\psi_{c}(x), \eqno \hspace {6cm}(3.53)$$
 for all $x\in X,$ where
\begin{align*}\psi_{c}(x)=&
\frac{1}{k^2(1-k^2)}~[(5-4k^2)\varphi_c(x,x)+k^2\varphi_c(2x,2x)\\&+2k^2\varphi_c(2x,x)
+(4-2k^2)\varphi_c(x,2x)+\varphi_c(x,3x)
+2\varphi_c((1+k)x,x)\\&+2\varphi_c((1-k)x,x)
+\varphi_c((1+2k)x,x)+\varphi_c((1-2k)x,x)~] , \hspace
{2.6cm}(3.54)\end{align*} for all $x \in X.$ Letting $h:X \to Y$
be a function defined by $h(x):=f(2x)-2f(x).$ Then, we conclude
that
$$\|h(2x)-8h(x)\|\leq M^5\psi_{c}(x) \eqno\hspace {7.8cm}(3.55)$$ for all $x \in X.$
If we replace $x$ in (3.55) $\frac{x}{2^{n+1}}$ and multiply both
sides of (3.55) by $8^n,$ we get
$$\|8^{n+1}h(\frac{x}{2^{n+1}})-8^nh(\frac{x}{2^n})\|_Y\leq M^5 8^n \psi_{c}(\frac{x}{2^{n+1}})
 \eqno\hspace {5.1cm}(3.56)$$
for all $x\in X$ and all non-negative integers $n$. Since $Y$ is
p-Banach space, then by (3.56), we have
\begin{align*}
\|8^{n+1}h(\frac{x}{2^{n+1}})-8^m h(\frac{x}{2^m})\|_Y^p&\leq
\sum^{n}_{i=m}\|8^{i+1}h(\frac{x}{2^{i+1}})-8^i
h(\frac{x}{2^i})\|_Y^p\\
&\leq M^{5p} \sum^{n}_{i=m} 8^{ip} {\psi_{c}}^p
(\frac{x}{2^{i+1}}) \hspace {4.4cm}(3.57)
\end{align*}
for all non-negative integers $n$ and $m$ with $n\geq m$ and all
$x\in X.$ Since $0< p\leq1$, then by Lemma 3.1, we get from
(3.54),
\begin{align*}{\psi_{c}}^p(x)\leq&
\frac{1}{k^{2p}(1-k^2)^p}~[(5-4k^2)^p{\varphi_c}^p(x,x)+k^{2p}{\varphi_c}^p(2x,2x)
\\&+(2k^2)^p{\varphi_c}^p(2x,x)
+(4-2k^2)^p{\varphi_c}^p(x,2x)+{\varphi_c}^p(x,3x)
+2^p{\varphi_c}^p((1+k)x,x)\\&+2^p{\varphi_c}^p((1-k)x,x)
+{\varphi_c}^p((1+2k)x,x)+{\varphi_c}^p((1-2k)x,x)~]  \hspace
{1.7cm}(3.58)\end{align*}
 for all $x\in X.$ Therefore  it follows from (3.48) and (3.58) that
$$\sum^{\infty}_{i=1} 8^{ip} {\psi_{c}}^p(\frac{x}{2^i})<\infty \eqno\hspace
{7cm}(3.59)\hspace {.2cm}$$ for all $x\in X.$ Therefore we
conclude from (3.57) and (3.59) that the sequence
$\{8^{n}h(\frac{x}{2^n})\}$ is a Cauchy sequence for all $x\in X.$
Since $Y$ is complete, the sequence $\{8^{n}h(\frac{x}{2^n})\}$
converges for all $x\in X.$ So one can define the function
$C:X\rightarrow Y$ by $$C(x)=\lim_{n \to
\infty}8^nh(\frac{x}{2^n}) \eqno\hspace {7cm}(3.60)\hspace
{.2cm}$$ for all $x\in X.$ Letting $m=0$ and passing the limit
$n\rightarrow\infty$ in (3.57), we get
$$\|h(x)-C(x)\|_Y^p\leq M^{5p}\sum^{\infty}_{i=0}8^{ip}{\psi_{c}}^p
(\frac{x}{2^{i+1}})
=\frac{M^{5p}}{8^p}\sum^{\infty}_{i=1}8^{ip}{\psi_{c}}^p(\frac{x}{2^i})
\eqno\hspace {2cm}(3.61)\hspace {.2cm}$$ for all $x\in X.$
Therefore, (3.51) follows from (3.48) and (3.61). Now we show that
$C$ is cubic. It follows from (3.47), (3.56) and (3.60) that
\begin{align*}
\|C(2x)-8C(x)\|_Y &=\lim_{n \to \infty}\|8^n
h(\frac{x}{2^{n-1}})-8^{n+1}h(\frac{x}{2^n})\|_Y\\
&= 8 \lim_{n \to \infty} \|8^{n-1}h(\frac{x}{2^{n-1}})-8^n
h(\frac{x}{2^n})\|_Y \\
& \leq M^5 \lim_{n \to \infty}8^n \psi_{c}(\frac{x}{2^n})=0
\hspace {6cm}\end{align*} for all $x \in X.$ So
$$C(2x)=8C(x)\eqno\hspace {9cm}(3.62)$$ for all $x \in X.$ On the
other hand it follows from (3.47) , (3.49) and (3.50) that
\begin{align*}
\|D_C(x,y)\|_Y&=\lim_{n \to \infty} 8^n
\|D_h(\frac{x}{2^{n}},\frac{y}{2^{n}})\|_Y=\lim_{n \to \infty}8^n
\|D_f(\frac{x}{2^{n-1}},\frac{y}{2^{n-1}})-2
D_f(\frac{x}{2^n},\frac{y}{2^n})\|_Y\\
&\leq M^5\lim_{n \to \infty}
8^n\{\|D_f(\frac{x}{2^{n-1}},\frac{y}{2^{n-1}})\|_Y+2
\|D_f(\frac{x}{2^n},\frac{y}{2^n})\|_Y\}\\
&\leq M^5\lim_{n \to \infty}
8^n\{\varphi_c(\frac{x}{2^{n-1}},\frac{y}{2^{n-1}})+2
\varphi_c(\frac{x}{2^n},\frac{y}{2^n})\}=0
\end{align*}
for all $x,y \in X.$ Hence the function $C$ satisfies $(1.5).$ By
Lemma 2.2, the function $x \rightsquigarrow C(2x)-8C(x)$ is
additive. Hence, (3.62) implies that function $C$ is cubic.\\
To prove the uniqueness of $C,$ let $C^{'}:X \to Y$ be another
additive function satisfying (3.51). Since $$\lim_{n \to
\infty}8^{np}\sum_{i=1}^{\infty}8^{ip}{\varphi_c}^p(\frac{x}{2^{n+i}},\frac{x}{2^{n+i}})
=\lim_{n\to\infty}\sum_{i=n+1}^{\infty}8^{ip}{\varphi_c}^p
(\frac{x}{2^i},\frac{x}{2^i})=0\eqno\hspace{2cm}$$ for all $x\in
X$ and for all $y\in \{x, 2x, 3x\},$ then
$$\lim_{n \to \infty}8^{np}\widetilde{\psi}_c(\frac{x}{2^n})=0\eqno\hspace{7.5cm}(3.63)$$
for all $x \in X$. It follows from (3.51) and (3.63) that
$$\|C(x)-C^{'}(x)\|_Y=\lim_{n \to \infty}8^{np}{\|h(\frac{x}{2^n})-C^{'}(\frac{x}{2^n})\|_Y}^p
\leq \frac{M^{5p}}{8^p}\lim_{n \to
\infty}8^{np}\widetilde{\psi}_c(\frac{x}{2^n})=0$$
for all $x \in X.$ So $C=C^{'}.$\\
For $j=-1$, we can prove the Theorem by a similar technique.
\end{proof}

\begin{cor}\label{t2}
Let $\theta, r, s$ be non-negative real numbers such that $r,s> 3$
or $0\leq r,s<3$. Suppose that an odd function $f:X\rightarrow Y$
satisfies the inequality (3.45)
 for all $x, y \in X.$ Then
there exists a unique cubic function $C:X\rightarrow Y$ satisfying
$$\| f(2x)-2f(x)-C(x)\|_Y \leq \frac{M^5 \theta}{k^2(1-k^2)}
\left\{%
\begin{array}{ll}
    \delta_{c}, & \hbox{r =s =0;} \\
    \alpha_{c}~\|x\|_X^r, & \hbox{r $>$ 0, s=0;} \\
    \beta_{c}~\|x\|_X^s, & \hbox{r=0, s $>$ 0;} \\
    (\alpha_{c}^p~\|x\|_X^{rp}+\beta_{c}^p~\|x\|_X^{sp})^{\frac{1}{p}}, & \hbox{r, s $>$ 0.} \\
\end{array}%
\right.\eqno\hspace{4cm}$$ for all $x \in X,$ where
$$\delta_{c}=\textbf{\{}~\frac{1}{8^p-1}~[(5-4k^2)^p+(4-2k^2)^p+k^{2p}(2^p+1)
+2^{p+1}+3]~\textbf{\}}^\frac{1}{p},\eqno\hspace{2cm}$$
\begin{align*}\alpha_{c}=\textbf{\{}~\frac{1}{|8^p-2^{rp}|}~[(5-4k^2)^p+(4-2k^2)^p+&(1+2k)^{rp}
+(1-2k)^{rp}+2^p(1+k)^{rp}\\&+2^p(1-k)^{rp}+2^{rp}k^{2p}(2^p+1)+1]~\textbf{\}}^\frac{1}{p}
,\end{align*}

\begin{align*}\beta_{c}=\textbf{\{}~\frac{1}{|8^p-2^{sp}|}~[(5-4k^2)^p+2^{sp}(4-2k^2)^p+
k^{2p}(2^{sp}+2^p)+3^{sp}+2^{p+1}+2]~\textbf{\}}^\frac{1}{p}
.\end{align*}

\end{cor}

\begin{proof}
In Theorem 3.7, let $\varphi_{a}(x, y):=\theta(\|x\|_X^r+\|y\|_X^s)$
for all $x, y \in X.$
\end{proof}

\begin{cor}\label{t2}
Let $\theta\geq0$ and $r, s>0$ be non-negative real numbers such
that $\lambda:=r+s\neq3$. Suppose that an odd function
$f:X\rightarrow Y$ satisfies the inequality (3.46) for all $x, y
\in X.$ Then there exists a unique cubic function
 $C:X\rightarrow Y$ satisfying
$$\|f(2x)-2f(x)-C(x)\|_Y \leq\frac{M^5\theta}{k^2(1-k^2)}~~\varepsilon_{c}~\|x\|_X^\lambda
,\eqno \hspace {5cm} $$ for all $x \in X,$ where
\begin{align*}\varepsilon_{c}=\textbf{\{}~\frac{1}{|8^p-2^{\lambda p}|}~[(5-4k^2)^p&+
2^{sp}(4-2k^2)^p+(1+2k)^{rp}+(1-2k)^{rp}+2^p(1+k)^{rp}
\\&+2^p(1-k)^{rp}+k^{2p}(2^{\lambda p}+2^{(r+1)p})+3^{sp}]~\textbf{\}}^\frac{1}{p}
\end{align*}
\end{cor}

\begin{proof}
In Theorem 3.7, let $\varphi_{c}(x, y):=\theta\|x\|_X^r \|y\|_X^s$
for all $x, y \in X.$
\end{proof}

\begin{thm}\label{t2}
Let $j\in \{-1,1\}$ be fixed and let $\varphi:X\times X\rightarrow
[0, \infty)$ be a function such that
\begin{align*}
\lim_{n\rightarrow\infty}\{&(\frac{1+j}{2})2^{nj}
\varphi(\frac{x}{2^{nj}},\frac{y}{2^{nj}})+(\frac{1-j}{2}) 8^{nj}
\varphi(\frac{x}{2^{nj}},\frac{y}{2^{nj}})\}
\\&=0=\lim_{n\rightarrow\infty}\{(\frac{1-j}{2})2^{nj}
\varphi(\frac{x}{2^{nj}},\frac{y}{2^{nj}})+(\frac{1+j}{2}) 8^{nj}
\varphi(\frac{x}{2^{nj}},\frac{y}{2^{nj}})\} \hspace
{2.5cm}(3.64)\end{align*} for all $x,y\in X$ and
\begin{align*}
\sum^{\infty}_{i=\frac{1+j}{2}} \{&(\frac{1+j}{2})2^{ipj}
\varphi^p(\frac{x}{2^{ij}},\frac{x}{2^{ij}})+(\frac{1-j}{2})8^{ipj}
\varphi^p(\frac{x}{2^{ij}},\frac{x}{2^{ij}})\}<\infty, \\&
\sum^{\infty}_{i=\frac{1+j}{2}} \{(\frac{1-j}{2})2^{ipj}
\varphi^p(\frac{x}{2^{ij}},\frac{x}{2^{ij}})+(\frac{1+j}{2})8^{ipj}
\varphi^p(\frac{x}{2^{ij}},\frac{x}{2^{ij}})\}<\infty \hspace
{2.2cm}(3.65)\end{align*}
 for all $x\in X$ and for all $y\in \{x, 2x, 3x\}$. Suppose that
an odd function $f:X\rightarrow Y$  satisfies the inequality
$$\|D_f(x,y)\|_Y \leq\varphi(x,y),\eqno\hspace{6.5 cm} (3.66)$$
for all $x,y\in X.$ Then there exist a unique additive function
$A:X \to Y$ and a unique cubic function $C:X \to Y$ such that
$$\|f(x)-A(x)-C(x)\|_Y\leq \frac{M^6}{48}~( 4[\widetilde{\psi}_a(x)]^{\frac{1}{p}}
+[\widetilde{\psi}_c(x)]^{\frac{1}{p}} )\eqno\hspace{2 cm}(3.67)$$
for all $x \in X,$ where $\widetilde{\psi}_a(x)$ and
 $\widetilde{\psi}_c(x)$ has been defined in (3.19) and (3.52),
respectively, for all $x \in X.$
\end{thm}
\begin{proof}
Let $j=1.$ By Theorem 3.4 and 3.7 , there exist an additive function
$A_0:X \to Y$ and a cubic function $C_0:X \to Y$ such that
$$~\|f(2x)-8f(x)-A_0(x)\|_Y
\leq\frac{M^5}{2}[\widetilde{\psi}_a(x)]^{\frac{1}{p}},
 \hspace{.8cm}\|f(2x)-2f(x)-C_0(x)\|_Y\leq\frac{M^5}{8}[\widetilde{\psi}_c(x)]^{\frac{1}{p}}$$
for all $x \in X.$ Therefore, it follows from the last inequality
that $$\|f(x)+\frac{1}{6}A_0(x)-\frac{1}{6}C_0(x)\|_Y \leq
\frac{M^6}{48}~( 4[\widetilde{\psi}_a(x)]^{\frac{1}{p}}
+[\widetilde{\psi}_c(x)]^{\frac{1}{p}} )\eqno\hspace{4.4 cm}$$ for
all $x \in X.$ So we obtain $(3.67)$ by letting
$A(x)=-\frac{1}{6}A_0(x)$ and $C(x)=\frac{1}{6}C_0(x)$ for all $x
\in X.$ To prove the uniqueness property of $A$ and $C,$ let
$A_1,C_1:X \to Y$ be another additive and cubic functions satisfying
(3.67). Let $A^{'}=A-A_1$ and $C^{'}=C-C_1.$ So
\begin{align*}
\|A^{'}(x)+C^{'}(x)\|_Y &\leq M\{\|f(x)-A(x)-C(x)\|_Y
+\|f(x)-A_1(x)-C_1(x)\|_Y \}\\&\leq\frac{M^7}{24}~(
4[\widetilde{\psi}_a(x)]^{\frac{1}{p}}
+[\widetilde{\psi}_c(x)]^{\frac{1}{p}} ) \hspace{5.1cm} (3.68)
\end{align*}
for all $x \in x.$ Since
$$\lim_{n \to \infty} 2^{np} \widetilde{\psi}_a(\frac{x}{2^n})=
\lim_{n \to
\infty}8^{np}\widetilde{\psi}_c\frac{x}{2^n})=0\eqno\hspace{5cm}$$
for all $x \in X.$ Then (3.68) implies that $$\lim_{n \to
\infty}8^n\|A^{'}(\frac{x}{2^n})+C^{'}(\frac{x}{2^n})\|_Y=0\eqno\hspace{5.8cm}$$
for all $x \in X.$ Therefore $C^{'}=0.$ So it follows from (3.68)
that
$$\|A^{'}(x)\|_Y \leq
\frac{5M^7}{24}[\widetilde{\psi}_a(x)]^{\frac{1}{p}}\eqno\hspace{6.5cm}$$
for all $x \in X.$ Therefore $A^{'}=0.$\\ For $j=-1$, we can prove
the Theorem by a similar technique.
\end{proof}

\begin{cor}\label{t2}
Let $\theta, r, s$ be non-negative real numbers such that $r,s> 3$
or $1< r,s <3$ or $0\leq r,s<1$. Suppose that an odd function
$f:X\rightarrow Y$ satisfies the inequality (3.45)
 for all $x, y \in X.$ Then
there exists a unique additive function $A:X\rightarrow Y$ and a
unique cubic function $C:X\rightarrow Y$ such that
$$\| f(x)-A(x)-C(x)\|_Y \leq \frac{M^6 \theta}{6k^2(1-k^2)}
\left\{%
\begin{array}{ll}
    \delta_{a}+\delta_{c}, & \hbox{r =s =0;} \\
    (\alpha_{a}+\alpha_{c})~\|x\|_X^r, & \hbox{r $>$ 0, s=0;} \\
    (\beta_{a}+\beta_{c})~\|x\|_X^s, & \hbox{r=0, s $>$ 0;} \\
    \gamma_{a}(x)+\gamma_{c}(x), & \hbox{r, s $>$ 0.} \\
\end{array}%
\right.\eqno\hspace{4cm}$$ for all $x \in X,$ where $\delta_{a},
\delta_{c}, \alpha_{a}, \alpha_{c}, \beta_{a}$ and $\beta_{c}$ are
defined as in Corollaries 3.5 and 3.8 and
$$
\gamma_{a}(x)=\{\alpha_{a}^p~\|x\|_X^{rp}+\beta_{a}^p~\|x\|_X^{sp}\}^\frac{1}{p},
\hspace{1cm}\gamma_{c}(x)=\{\alpha_{c}^p~\|x\|_X^{rp}+\beta_{c}^p~\|x\|_X^{sp}\}^\frac{1}{p}$$
 for all $x \in X.$
\end{cor}

\begin{cor}\label{t2}
Let $\theta\geq0$ and $r, s>0$ be non-negative real numbers such
that $\lambda:=r+s \in (0,1)\cup(1,3)\cup(3,\infty)$. Suppose that
an odd function $f:X\rightarrow Y$ satisfies the inequality (3.46)
for all $x, y \in X.$ Then there exist  a unique additive function
$A:X\rightarrow Y$ and a unique cubic function $C:X\rightarrow Y$
such that
$$\|f(x)-A(x)-C(x)\|_Y \leq\frac{M^6\theta}{6k^2(1-k^2)}~~
(\varepsilon_{a}+\varepsilon_{c})~\|x\|_X^\lambda ,\eqno \hspace
{5cm} $$ for all $x \in X,$ where $\varepsilon_{a}$ and
$\varepsilon_{c}$ are defined as in Corollaries 3.6 and 3.9.
\end{cor}

\begin{thm}\label{t2}
Let $\varphi:X\times X\rightarrow [0, \infty)$ be a function which
satisfies (3.1) for all $x,y\in X$ and (3.2) for all $x\in X$ or
satisfies (3.64) for all $x,y\in X$ and (3.65) for all $x\in X$
 and for all $~y\in \{x,2x, 3x\}.$
Suppose that a function $f:X\rightarrow Y$ with $f(0)=0$ satisfies
the inequality (3.3) for all $x,y\in X.$ Then there exist a unique
additive function $A:X \to Y,$ a unique quadratic function $Q:X
\to Y,$ and a unique cubic function $C:X \to Y$ such that
\begin{align*}
\|f(x)-A(x)-Q(x)-C(x)\|_Y
\leq\frac{M^8}{96}&\textbf{\{}4[\widetilde{\psi}_a(x)+\widetilde{\psi}_a(-x)]^{\frac{1}{p}}
+[\widetilde{\psi}_c(x)+\widetilde{\psi}_c(-x)]^{\frac{1}{p}}\textbf{\}}\\&+
\frac{M^3}{4k^2}\textbf{\{}[\widetilde{\psi_e}(x)+\widetilde{\psi_e}(-x)]^{\frac{1}{p}}\textbf{\}}
 \hspace{2.6cm}(3.69)
\end{align*}
 for all $x \in X,$ where $\widetilde{\psi_e}(x),
 \widetilde{\psi}_a(x)$ and $\widetilde{\psi}_c(x)$ have been
 defined in (3.2), (3.19) and (3.52), respectively, for all $x \in X.$
\end{thm}
\begin{proof}
Assume that $\varphi:X\times X\rightarrow [0, \infty)$
 satisfies (3.1) for all $x,y\in X$ and (3.2) for
all $x\in X.$ Let $f_e(x)=\frac{1}{2}(f(x)+f(-x))$ for all $x\in
X.$ Then $f_e(0)=0,$ $f_e(-x)=f_e(x)$ and
$$\|D_{f_e}(x,y)\|\leq\frac{M}{2}[\varphi(x,y)+\varphi(-x,-y)]\eqno
\hspace {5.2cm}$$ for all $x,y\in X.$ Hence, from Theorem 3.2,
there exists a unique quadratic function $Q:X\rightarrow Y$
satisfying
$$\|f_{e}(x)-Q(x)\|_{Y}\leq \frac{M}{2k^2}
~[\widetilde{\psi_e}(x)]^\frac{1}{p}.\eqno \hspace {6.4cm}(3.70)$$
for all $x\in X.$ It is clear that
$$\widetilde{\psi_e}(x)\leq \frac{M^p}{2^p}
~[\widetilde{\psi_e}(x)+\widetilde{\psi_e}(-x)],\eqno \hspace
{7.6cm}$$ for all $x\in X.$ Therefore it follows from (3.70) that
$$\|f_{e}(x)-Q(x)\|_{Y}\leq \frac{M^2}{4k^2}
~[\widetilde{\psi_e}(x)+\widetilde{\psi_e}(-x)]^\frac{1}{p}.\eqno
\hspace {5cm}(3.71)$$ Let $f_o(x)=\frac{1}{2}(f(x)-f(-x))$ for all
$x\in X.$ Then $f_o(0)=0,$ $f_o(-x)=-f_o(x)$ and
$$\|D_{f_o}(x,y)\|\leq\frac{M}{2}[\varphi(x,y)+\varphi(-x,-y)]\eqno
\hspace {6.4cm}$$ for all $x,y\in X.$ By Theorem 3.10,  there exist
a unique additive function $A:X \to Y$ and a unique cubic function
$C:X \to Y$ satisfy
$$\|f_{o}(x)-A(x)-C(x)\|_{Y}\leq \frac{M^6}{48}~( 4[\widetilde{\psi}_a(x)]^{\frac{1}{p}}
+[\widetilde{\psi}_c(x)]^{\frac{1}{p}} )\eqno \hspace
{3.3cm}(3.72)$$ for all $x\in X.$ Since
$$\widetilde{\psi_a}(x)\leq \frac{M^p}{2^p}
~[\widetilde{\psi_a}(x)+\widetilde{\psi_a}(-x)],\hspace{1cm}
\widetilde{\psi_c}(x)\leq\frac{M^p}{2^p}
~[\widetilde{\psi_c}(x)+\widetilde{\psi_c}(-x)]\eqno
\hspace{1cm}$$ for all $x\in X.$ Therefore it follows from (3.72)
that
$$\|f_{o}(x)-A(x)-C(x)\|_{Y}\leq \frac{M^7}{96}~\textbf{\{}4[\widetilde{\psi}_a(x)
+\widetilde{\psi}_a(-x)]^{\frac{1}{p}}
+[\widetilde{\psi}_c(x)+\widetilde{\psi}_c(-x)]^{\frac{1}{p}}\textbf{\}}
\eqno \hspace {.8cm}(3.73)$$
 for all $x \in X.$ Hence (3.69) follows from (3.71) and (3.73).
Now, if $\varphi:X\times X\rightarrow [0, \infty)$ satisfies
(3.64) for all $x,y\in X$ and (3.65) for all $x\in X$
 and for all $y\in \{x,2x, 3x\},$ we can prove the Theorem by a
 similar technique.
\end{proof}

\begin{cor}\label{t2}
Let $\theta, r, s$ be non-negative real numbers such that $r,s> 3$
or $2< r,s <3$ or $1< r,s <2$ or $0< r,s<1$. Suppose that a
function $f:X\rightarrow Y$ with $f(0)=0$ satisfies the inequality
(3.13)
 for all $x, y \in X.$ Then
there exist a unique additive function $A:X\rightarrow Y$ and a
unique quadratic function $Q:X\rightarrow Y$ and a unique cubic
function $C:X\rightarrow Y$ such that
\begin{align*}\| f(x)-A(x)-Q(x)-C(x)\|_Y \leq &\frac{M^8
\theta}{6k^2(1-k^2)}\textbf{\{}[\alpha_{a}^p~\|x\|_X^{rp}+\beta_{a}^p~\|x\|_X^{sp}]^\frac{1}{p}
\\&\hspace{.2cm}+[\alpha_{c}^p~\|x\|_X^{rp}+\beta_{c}^p~\|x\|_X^{sp}]^\frac{1}{p}\textbf{\}}
+\frac{M^3 \theta}{2}
[\frac{1}{|k^{2p}-k^{sp}|}~\|x\|_X^{sp}]^\frac{1}{p}
\hspace{1cm}\end{align*}
 for all $x \in X,$ where $ \alpha_{a}, \alpha_{c}, \beta_{a}$ and $\beta_{c}$ are
defined as in Corollaries 3.5 and 3.8.
\end{cor}
\begin{proof}
 Put $\varphi(x,y):=\theta(\|x\|_X^r+\|y\|_X^s)$, since $$\|D_{f_e}(x,y)\|\leq
M\varphi(x,y),\hspace {1cm} \|D_{f_o}(x,y)\|\leq
M\varphi(x,y)\eqno \hspace {3cm}$$ for all $x, y \in X.$ Thus the
result follows from Corollaries 3.3 and 3.11.
\end{proof}

{\small

%----------------------------------------------------------------------%

}
\end{document}